\documentclass[11pt, reqno]{amsart}
\usepackage{version}
\usepackage{amssymb}
\usepackage{enumitem}

\theoremstyle{definition}
\newtheorem{thm}{Theorem}[section]

\newtheorem{lem}{Lemma}[section]
\newtheorem{prop}{Proposition}[section]

\newtheorem{conj}{Conjecture}[section]

\usepackage{hyperref}
\usepackage{lipsum}

\newcommand{\myfootnote}[1]{
    \renewcommand{\thefootnote}{}
    \footnotetext{\hspace{-2pt}\scriptsize#1}
    \renewcommand{\thefootnote}{\arabic{footnote}}
}

\newcommand*{\doi}[1]{\href{http://dx.doi.org/#1}{DOI: #1}}

\usepackage{amsmath, amsfonts, amssymb,amsthm}
\newcommand{\field}[1]{\mathbb{#1}}

\newcommand{\C}{\field{C}}
\newcommand{\B}{\field{B}}

\DeclareMathOperator{\id}{Id}

\DeclareMathOperator{\dist}{dist}

\numberwithin{equation}{section}

\begin{document}
\title[]{Proper holomorphic embeddings of complements of large Cantor sets in $\mathbb C^2$.}
\author[Di Salvo]{G. D. Di Salvo}
\author[Wold]{E. F. Wold}
\address{E. F. Wold: Department of Mathematics\\
University of Oslo\\
Postboks 1053 Blindern, NO--0316 Oslo, Norway}\email{erlendfw@math.uio.no}

\address{G. D. Di Salvo: Department of Mathematics\\
University of Oslo\\
Postboks 1053 Blindern, NO--0316 Oslo, Norway}\email{giovandd@math.uio.no}

%
%
\date{\today}
\myfootnote
{
    Published in \emph{Arkiv f\"or Matematik},
    October~2022,
    volume~60,
    issue~2,
    pp.~323--333.
}
\myfootnote
{
    \doi{10.4310/ARKIV.2022.v60.n2.a5}.
}
\begin{abstract}
We presents a construction of a proper holomorphic embedding $f\colon\Bbb P^1\setminus C\hookrightarrow\C^2$, where $C$ is a Cantor set obtained by removing smaller and smaller vertical and horizontal strips from a square of side 2, allowing to realize it to have Lebesgue measure arbitrarily close to 4. 
\end{abstract}

\maketitle
\section{Introduction}

\subsection{The main result}
A major unresolved issue, known as Forster's Conjecture, is whether or not every open Riemann surface $X$ admits a holomorphic embedding into $\mathbb C^2$, and, if it does, whether it admits a \emph{proper} holomorphic embedding. For instance, if $Y$ is a compact Riemann surface, and if $X=Y\setminus C$ where $C$ is a closed set whose connected components are all points, it is unknown whether $X$ embeds (properly or not) into $\mathbb C^2$. We may consider two extremal cases: (i) the case where $C$ is a finite set, and (ii) the case where $C$ is a Cantor set, and we may further consider the simplest compact Riemann surface in this context, namely $Y=\mathbb P^1$. Then in case (i), it is clear that $X$ admits a proper holomorphic embedding into $\mathbb C^2$, so we will consider the case (ii). 

Let $Q$ denote the square $Q=[-1,1]\times [-1,1]\subset\mathbb C$, and let $\mu$ denote the 2--dimensional 
Lebesgue measure on $\mathbb C$. A procedure for constructing a (large) Cantor set $C\subset Q$ is 
as follows (see Section \ref{firstpaper:sec:ind} for a more detailed description). Let $l_1$ denote the vertical 
line dividing $Q$ into two equal pieces, choose $\delta_1>0$ small, and remove an open $\delta_1$--neighbourhood 
of $l_1$ to obtain a union $Q_2$ of two disjoint rectangles. Next, let $l_2^{j}, j=1,2$, be horizontal lines dividing 
each rectangle in $Q_2$ into equal pieces, choose $\delta_2$ small, and remove an open $\delta_2$--neighbourhood 
of $l_2^1\cup l_2^2$ to obtain a disjoint union $Q_3$ of four rectangles. Next, switch back to vertical lines, chose 
$\delta_3$ small to obtain $Q_4$ and so forth, to obtain a sequence $\delta_j\rightarrow 0$ and nested sequence $Q_j$ of rectangles 
such that $C=\cap_j Q_j$ is a Cantor set contained in $Q$. Our main result is the following.

\begin{thm}\label{firstpaper:main}
There are sequences $\{\delta_j\}_j$ converging to zero arbitrarily fast such that the complement $\Bbb P^1\setminus C$ of the resulting Cantor set admits a proper holomorphic embedding into $\Bbb C^2$.
In particular, for any $\epsilon>0$ we may achieve that $\mu(C)>4-\epsilon$.  
\end{thm}
The motivation for proving this result is that there have been speculations that considering complements of 
"fat" Cantor sets could lead to counterexamples to Forster's Conjecture.   

We note that Orevkov \cite{firstpaper:Orevkov2008} showed the existence of a Cantor set $C\subset\mathbb P^1$
such that $\mathbb P^1\setminus C$ admits a proper holomorphic embedding into $\mathbb C^2$. His
construction is quite cryptical and it is explained in detail in \cite{firstpaper:DS22}, where it is also proved that $C$ can be obtained to have zero Hausdorff dimension. From such a construction it seems difficult, or perhaps impossible, to achieve that $C$ is large.

\subsection{History}
Dealing with Stein manifolds, one of the most important goals to achieve is to embed them properly holomorphically into $\Bbb C^N$ for some $N$.
A first important result comes from Remmert \cite{firstpaper:R56}, who proved in 1956 that every $n$--dimensional Stein manifold admits a proper holomorphic embedding into $\C^N$ for $N$ big enough. Such a result was made precise by Bishop and Narasimhan who independently proved in 1960--61 that $N$ can be taken to be $2n+1$ (see \cite{firstpaper:B61} and \cite{firstpaper:N60}).
In 1970 Forster \cite{firstpaper:F70} improved Bishop--Narasimhan's result, decreasing $N$ to $\lfloor\frac{5n}{3}\rfloor+2$ and proving that it is not possible for $N$ to go below $\lfloor\frac{3n}{2}\rfloor+1$ and conjecturing that the euclidean dimension could have been improved exactly to  $\lfloor\frac{3n}{2}\rfloor+1$.
Eliashberg, Gromov (1992) and Sh\"urmann (1997) proved the following
\begin{thm}{}[Eliashberg--Gromov \cite{firstpaper:EG92} (1992) and Sh\"urmann \cite{firstpaper:Sh97} (1997)] Every $n$--dimensional Stein manifold $X$, with $n\ge2$ embeds properly holomorphically into $\C^N$ with $N=\lfloor\frac{3n}{2}\rfloor+1$.
\end{thm}

The proof of the theorem breaks down when $n=1$; since $1$--dimensional Stein manifolds are precisely open connected Riemann surfaces, Forster's conjecture reduces to the following 
\begin{conj}
Every open connected Riemann surface embeds properly holomorphically into~$\C^2$.
\end{conj}

So far, only a few open Riemann surfaces are known to admit a proper holomorphic embedding into $\C^2$. The first known examples are the open unit disk in $\C$ (Kasahara--Nishino, 1970, \cite{firstpaper:S70}), open annuli in $\C$ (Laufer, 1973, \cite{firstpaper:L73}) and the punctured disk in $\C$ (Alexander, 1977, \cite{firstpaper:A77}).
Later (1995) Stens\o nes and Globevnik proved in \cite{firstpaper:GS95} that every finitely connected planar domain without isolated boundary points verifies the conjecture.
In 2009, Wold and Forstneri\v c proved the best result known so far: if $\overline D$ is a Riemann surface with smooth enough boundary which admits a smooth embedding into $\C^2$, holomorphic on the interior $D$, then $D$ admits a proper holomorphic embedding into $\C^2$ (see \cite{firstpaper:ForstnericWold2009} and next section).
Other remarkable results include proper holomorphic embeddings of certain Riemann surfaces into $\C^2$ with interpolation (see \cite{firstpaper:KLW09}), deformation of continuous mappings $f\colon S\to X$ between Stein manifolds into proper holomorphic embeddings under certain hypothesis on the dimension of the spaces (see \cite{firstpaper:AFRW16}), embeddings of infinitely connected planar domains into $\C^2$ (see \cite{firstpaper:ForstnericWold2013}), the existence of a homotopy of continuous mappings $f\colon D\to\C\times\C^*$ into proper holomorphic embedding whenever $D$ is a finitely connected planar domain without punctures (see \cite{firstpaper:R13}), existence of proper holomorphic embeddings of the unit disc $\B$ into connected pseudoconvex Runge domains $\Omega\subset\C^n$ (when $n\ge2$) whose image contains arbitrarily fixed discrete subsets of $\Omega$ (see \cite{firstpaper:FGS96}), approximation of proper embeddings on smooth curves contained in a finitely connected planar domain $D$ into $\C^n$ (with $n\ge2$) by proper holomorphic embeddings $f\colon D\hookrightarrow\C^n$ (see \cite{firstpaper:M13}), and the existence of proper holomorphic embeddings into $\C^2$ of certain infinitely connected domains $\Omega$ lying inside a bordered Riemann surface $\overline D$ knowing to admit a proper holomorphic embedding into $\C^2$ \cite{firstpaper:M09}.

\section{Preliminaries}

\subsection{Notation}
We will use the following notation. 
\begin{itemize}
\item[$\bullet$] Given $K\subset\mathbb C$ and a positive real number $\delta$, we define the open subset
\begin{align*}
K({\delta})&:=\{z\in\C\;:\;\dist(z,K)<\delta\}.\\
\end{align*}
\item[$\bullet$] For a closed subset $K\subset\mathbb P^1$ we denote by $\mathcal O(K)$ the algebra of continuous functions 
$f\in\mathcal C(K)$ such that there exists an open set $U\subset\mathbb P^1$ containing $K$, and $F\in\mathcal O(U)$ with $F|_K=f$.

\item[$\bullet$]We let $\pi_j\colon\mathbb C^2\rightarrow\mathbb C$ denote the projection onto the $j$--th coordinate line, and given a point $p\in\C^2$ we denote the vertical complex line through $p$ by
$$
\Lambda_p:=\pi_1^{-1}(\pi_1(p))=\{(\pi_1(p),\zeta)\;:\;\zeta\in\C\}\;.
$$
\item[$\bullet$] If $X$ is a domain with piecewise smooth boundary in a Riemann surface $Y$, $f\colon \overline X\rightarrow\mathbb C^2$ is a holomorphic
map, and if $a\in\partial X$ is a smooth boundary point, we say that $f(a)$ is $\pi_1$--exposed for $f(\overline X)$ if  
$f(\overline X)\cap \Lambda_{f(a)}=\{f(a)\}$, and $\pi_1\circ f$ is an embedding sufficiently close to $a$. 
Similarly, for a smooth map $\gamma\colon [0,1]\rightarrow\mathbb C^2$, we say that $\gamma([0,1])$ is exposed at $\gamma(1)$
if $\gamma([0,1])\cap \Lambda_{\gamma(1)}=\{\gamma(1)\}$, and $\pi_1\circ \gamma$ is an embedding sufficiently close to $1$.
\end{itemize}

\subsection{Results}
In this section, we collect the technical tools needed to prove Theorem \ref{firstpaper:main}. \

\medskip

The following result is essentially Theorem 4.2 in \cite{firstpaper:ForstnericWold2009}. Although (1) and (2) were not stated explicitly in \cite{firstpaper:ForstnericWold2009}
they are evident from the proof therein and were added to the corresponding Theorem 2.8 in \cite{firstpaper:ForstnericWold2013}. 
\begin{thm}\label{firstpaper:exposethm}
Let $X$ be a smoothly bounded domain in a Riemann surface $Y$, $f\colon \overline X\hookrightarrow\mathbb C^2$ a holomorphic embedding, 
and $a_1,\dots,a_m\in\partial X$.  Let $\gamma_j\colon [0,1]\rightarrow\mathbb C^2$ $(j=1,\dots,m)$ be smooth embedded arcs with pairwise
disjoint images satisfying the following properties:
\begin{itemize}
\item[$\bullet$] $\gamma_j([0,1])\cap f(\overline X)=\gamma_j(0)=f(a_j)$ for $j=1,\dots,m$, and
\item[$\bullet$] the image $E:=f(\overline X)\cup\bigcup_{j=1}^m\gamma_j([0,1])$ is $\pi_1$--exposed at $\gamma_j(1)$ for $j=1,\dots,m$. 
\end{itemize}
Then given an open set $V\subset\mathbb C^2$ containing $\bigcup_{j=1}^m\gamma_j([0,1])$, an open set $U\subset Y$
containing the points $a_j$ that satisfies $f(\overline{U\cap X})\subset V$, and 
any $\epsilon>0$, there exists a holomorphic embedding $F\colon \overline X\hookrightarrow\mathbb C^2$ with the following properties:
\begin{enumerate}[label=(\arabic*)]
\item $\|F-f\|_{\overline X\setminus U}<\epsilon$, \label{firstpaper:q}
\item $F(\overline{U\cap X})\subset V$, and
\item $F(a_j)=\gamma_j(1)$, and $F(\overline X)$ is $\pi_1$--exposed at $F(a_j)$ for $j=1,\dots,m$.
\end{enumerate}
\end{thm}
The following is essentially Lemma 1 in \cite{firstpaper:Wold2006}. The difference is that Lemma 1 was stated for $\pi_1$ instead of $\pi_2$, and 
for curves $\lambda\colon [0,+\infty)\rightarrow\mathbb C^2$ instead of  $\lambda\colon (-\infty,+\infty)\rightarrow\mathbb C^2$ -- neither make a difference for the proof. 

\begin{lem}\label{firstpaper:rmcurves}
Let $K\subset\mathbb C^2$ be a polynomially convex compact set, and let 
$$
\Lambda=\{\lambda_j(t):j=1,\dots,m,\;\; t\in (-\infty,+\infty)\}
$$ 
be a collection of pairwise disjoint smooth curves in
$\mathbb C^2\setminus K$ without self intersection, enjoying the \emph{immediate projection property} (with respect to $\pi_2$):
\begin{itemize}
\item $\lim_{|t|\rightarrow\infty}|\pi_2(\lambda_j(t))|=\infty$ for all $j$, and
\item there exists an $M>0$ such that $\mathbb C\setminus(R\, \overline\B\cup\pi_2(\Lambda))$ does not 
contain any relatively compact components for $R\geq M$.
\end{itemize}
Then for any $r>0$ and $\epsilon>0$ there exists $\phi\in\operatorname{Aut}\mathbb C^2$
such that the following are satisfied: 
\begin{enumerate}[label=(\roman*)]
\item $\|\phi-\id\|_K<\epsilon$\,, and 
\item $\phi(\Lambda)\subset\mathbb C^2\setminus r\overline{\mathbb B^2}$.
\end{enumerate}
 
\end{lem}

\section{The Induction Step}\label{firstpaper:sec:ind}

We will now describe an inductive procedure to construct a nested sequence of closed rectangles $Q_n\subset Q$,
along with holomorphic embeddings $f_n\colon \overline{\mathbb P^1\setminus Q_n}\rightarrow\mathbb C^2$ that will 
be used to construct a proper holomorphic embedding 
$$
f\colon\mathbb P^1\setminus \bigcap_nQ_n\hookrightarrow\mathbb C^2,
$$ 
where $C=\bigcap_nQ_n$ will be a Cantor set, where the construction will enable us to ensure that its Lebesgue measure $\mu(C)$
is arbitrarily close to 4. \

\medskip

Set $Q_1:=Q$ and set $C_1:= \overline{\mathbb P^1\setminus Q_1}$. To construct $Q_2$ from $Q_1$ we let $l_1$  be the vertical line segment dividing $Q_1$ into two equal pieces. 
Then, for $0<\delta_2<<1$, we set 
$$
Q_2:=Q_1\setminus l_1(\delta_2), 
$$
and we set $C_2:=\overline{\mathbb P^1\setminus Q_2}$.
Then $C_2$ is the complement of the disjoint union of $2$ open rectangles $(Q_2^{j})^\circ,\;j=1,2$, contained in $Q_1$.\

\medskip

Assume now that we have constructed a nested sequence $\{Q_j\}_{j=1}^n$, $n\geq 2$, where 
$$
Q_n=\bigsqcup_{j=1}^{2^{n-1}} Q_n^{j}
$$ 
is the disjoint union of $2^{n-1}$ closed  rectangles contained in $Q_{n-1}$, along with an increasing sequence of closed subsets $C_n:=\overline{\Bbb P^1\setminus Q_n}$ in $\mathbb P^1$.
We let $l_n^{j}$ be the line segment -- vertical for $n$ odd, horizontal for $n$ even -- dividing $Q_n^{j}$ into two equal pieces, 
we set $l_n:=\bigsqcup_{j=1}^{2^{n-1}}l_n^{j}$, and for $\delta_{n+1}>0$ small enough we define 
$$
Q_{n+1}:= Q_n\setminus l_n(\delta_{n+1}) =: \bigsqcup_{j=1}^{2^n} Q_{n+1}^{j},
$$
and
$$
C_{n+1}:=\overline{\mathbb P^1\setminus Q_{n+1}}\;.
$$

\begin{prop}\label{firstpaper:indstep}
With the procedure above assume that we have constructed $Q_n$ and $C_n$ for $n\geq 1$. 
Let $K_n\subset C_n^\circ$ be a compact set, let $r_n>0$, and  
assume that $f_n\colon C_n\hookrightarrow\C^2$ is a holomorphic embedding such that
\begin{align}\label{firstpaper:1}
f_n(\overline{C_n\setminus K_n})\subset\C^2\setminus r_{n}\overline{\B^2}\;.
\end{align}
Then for any $\varepsilon_n>0$ and any $r_{n+1}>r_n$, there exist $\delta_{n+1}>0$ arbitrarily close to zero 
and a holomorphic embedding $f_{n+1}\colon C_{n+1}\hookrightarrow\C^2$ such that
\begin{enumerate}[label=(\alph*)]
\item $\|f_{n+1}-f_n\|_{K_n}<\varepsilon_n$\;,\label{firstpaper:prop:one}\
\item $f_{n+1}(\overline{C_{n+1}\setminus K_n})\subset \C^2\setminus r_{n}\overline{\B^2}$\;,\label{firstpaper:prop:two}
\item $f_{n+1}(\partial C_{n+1})\subset\C^2\setminus r_{n+1}\overline{\B^2}$\;.\label{firstpaper:prop:three}
\end{enumerate}
\end{prop}
\begin{proof}
We extend $f_n$ to a smooth embedding $\widetilde f_n\colon C_n\cup l_n\hookrightarrow\C^2$ with $\widetilde f_n(l_n)$ lying close enough to $f_n(\partial C_n)$ so that by $\eqref{firstpaper:1}$ 
we get
\begin{align}\label{firstpaper:2}
\widetilde f_n(l_n)\subset\C^2\setminus r_{n}\overline{\B^2}\;.
\end{align}
Now Mergelyan's theorem (see e.g., \cite{firstpaper:FornaessForstnericWold}) ensures the existence of a holomorphic embedding $\widehat f_{n+1}\colon C_n\cup l_n\hookrightarrow\mathbb C^2$ such that
$$
\|\widehat f_{n+1} - \widetilde f_n\|_{C_n\cup l_n}<\frac{\epsilon_n}{4},
$$
and 
\begin{align}
\widehat f_{n+1}((\overline{C_n\cup l_n)\setminus K_n})\subset\C^2\setminus r_{n}\overline{\B^2}\;.
\end{align}

Then by choosing a preliminary $\tilde\delta_{n+1}>0$ sufficiently small (to be shrunk further later), and letting the set corresponding to $C_{n+1}$
be denoted by $\tilde C_{n+1}$ (and similarly for $Q_{n+1}$), we have that $\widehat f_{n+1}\in\mathcal O(\tilde C_{n+1})$, and 
\begin{align}\label{firstpaper:4}
\widehat f_{n+1}(\overline{\tilde C_{n+1}\setminus K_n})\subset \C^2\setminus r_{n}\overline{\B^2}\;.
\end{align}

Next, recall that $\tilde Q_{n+1}$ is constructed from $Q_n$
by splitting each $Q_n^{j}$ into two smaller rectangles $\tilde Q_n^{j,1}$ and $\tilde Q_n^{j,2}$, by removing 
the strip $l_n^{j}(\tilde\delta_{n+1})$. Choose smooth boundary points 
\begin{equation}\label{firstpaper:bdchoose}
\tilde a_{ji}\in\partial \tilde Q_n^{j,i}\cap \overline{l_n^{j}(\tilde\delta_{n+1})},
\end{equation}
for $i=1,2$, and $j=1,\dots,2^{n-1}$, and relabel these to get $2^n$ boundary points $a_j$, one in each $\partial\tilde Q_{n+1}^{j}$. \

Now choose $2^n$ pairwise disjoint smoothly embedded arcs $\gamma_j\colon [0,1]\hookrightarrow\mathbb C^2$ disjoint from $r_n\overline{\mathbb B^2}$, such that 
$$
\gamma_j([0,1])\cap \widehat f_{n+1}(\tilde C_{n+1}) = \widehat f_{n+1}(a_j)=\gamma_j(0),
$$
and such that each point $\gamma_j(1)$ is $\pi_1$--exposed for the surface
$$
\widehat f_{n+1}(\tilde C_{n+1})\cup\bigcup_{j=1}^{2^n}\gamma_j([0,1]).
$$ 
Choose an open set $V\subset\C^2$ containing the arcs $\gamma_j([0,1])$ with $\overline V\cap r_n\overline{\mathbb B^2}=\emptyset$ and take $U\subset\Bbb P^1$ to be the union of sufficiently small open balls centered at the points $a_j$, so that $U\cap K_n=\emptyset$ and $\widehat f_{n+1}(\overline{U\cap\tilde C_{n+1}})\subset V$. Then Theorem \ref{firstpaper:exposethm} furnishes a holomorphic embedding $F_{n+1}\colon \tilde C_{n+1}\hookrightarrow\mathbb C^2$ 
such that $p_j:=F_{n+1}(a_j)=\gamma_j(1)$ is an exposed point for $F_{n+1}(\tilde C_{n+1})$ for each $j$, and 
\begin{align}\label{firstpaper:5}
\|F_{n+1}-\widehat f_{n+1}\|_{K_{n}}<\frac{\varepsilon_n}4,
\end{align}
and also 
\begin{align}\label{firstpaper:6}
F_{n+1}(\overline{\tilde C_{n+1}\setminus K_{n}})\subset\C^2\setminus r_{n}\overline{\B ^2}\;.
\end{align}
Now choose $\alpha_j\in\mathbb C$, $j=1,\dots,2^n$, such that setting 
$$
g_{n+1}(z,w):=\left(z,w+\sum_{j=1}^{2^n}\frac{\alpha_{j}}{\pi_1(p_{j})-z}\right)\;,
$$
we have that 
$$
\|g_{n+1}\circ F_{n+1}-F_{n+1}\|_{K_n}<\frac{\epsilon_n}{4},
$$
and such that the conditions in Lemma \ref{firstpaper:rmcurves} are satisfied for the collection $\Lambda$ of curves
$$
\lambda_{ji} := g_{n+1}\circ F_{n+1}(\partial \tilde Q_{n}^{j,i}),\;\;\;i=1,2,\;j=1,\dots,2^{n-1}\;,
$$
that 
are the boundary of the unbounded complex curve
$$
X_{n+1}:=g_{n+1}\circ F_{n+1}(\tilde C_{n+1})\;.
$$
Note that we still have 
\begin{align}\label{firstpaper:8}
g_{n+1}\circ F_{n+1}(\overline{\tilde C_{n+1}\setminus K_{n}})\subset\C^2\setminus r_{n}\overline{\B ^2}\;.
\end{align}
Choose $0<\eta<<1$ such that $(r_n+\eta)\overline{\mathbb B^2}\cap\Lambda=\emptyset$. We may choose a compact polynomially convex set $K'\subset X_{n+1}^\circ$  with $g_{n+1}\circ F_{n+1}(K_n)\subset K'$ such that $L=(r_n+\eta)\overline{\mathbb B^2}\cup K'$ is polynomially convex (see e.g., Theorem 4.14.6 in \cite{firstpaper:Forstnericbook}).
Then by Lemma \ref{firstpaper:rmcurves} there exists $\phi_{n+1}\in\operatorname{Aut}\C^2$ such that
\begin{align}\label{firstpaper:9}
\phi_{n+1}(\Lambda)\subset\C^2\setminus r_{n+1}\overline{\B ^2}\;,
\end{align}
and 
\begin{align}\label{firstpaper:10}
\|\phi_{n+1}-\id\|_{L}<\frac{\min\{\eta,\varepsilon_n\}}4\;.
\end{align}
\newline
We consider the map $f_{n+1}\colon \tilde C_{n+1}\rightarrow\mathbb C^2$ defined by
$$
f_{n+1}:=\phi_{n+1}\circ g_{n+1}\circ F_{n+1}\;.
$$
We have that \ref{firstpaper:prop:one} and \ref{firstpaper:prop:two} (with $\tilde C_{n+1}$ instead of $C_{n+1}$) clearly hold, but now $f_{n+1}$ has singularites on $\partial\tilde C_{n+1}$.
However, we now consider $0<\delta_{n+1}<\tilde\delta_{n+1}$ to see what happens on $C_{n+1}$. As the points to expose are taken on the boundary components (see \eqref{firstpaper:bdchoose}), the singularities of $f_{n+1}$ are not contained in $C_{n+1}$ for any such $\delta_{n+1}$, 
and so $f_n\colon C_{n+1}\rightarrow\mathbb C^2$ is holomorphic.  Finally, since $\partial C_{n+1}$ will converge to $\partial\tilde C_{n+1}$ as
$\delta_{n+1}\rightarrow\tilde\delta_{n+1}$ we have \ref{firstpaper:prop:three} for $\delta_{n+1}$ sufficiently close to $\tilde\delta_{n+1}$.

\end{proof}

\section{Proof of Theorem \ref{firstpaper:main}}

We will prove Theorem \ref{firstpaper:main} via an inductive construction, where Proposition \ref{firstpaper:indstep} provides us with the inductive step.  Without loss 
of generality, we may assume that $\epsilon<1$.  \

\subsection{The Induction Scheme}

To start the induction, with the notation as in Section \ref{firstpaper:sec:ind}, we define $f_1\colon C_1\hookrightarrow\mathbb C^2$ by 
$f_1(\zeta):=(2/\zeta,0)$ for $\zeta\in\mathbb C$, and $f(\infty):=(0,0)$. Setting $r_1=1$ we note that $f_1(\partial C_1)\subset\mathbb C^2\setminus r_1\overline{\mathbb B^2}$, so 
if we choose $0<\delta_1'<<1$ sufficiently close to zero, and set 
$$
K_1:=\mathbb P^1\setminus Q_1(\delta_1'), 
$$
we have that $K_1\subset C_1^\circ$ and $f_1(\overline{C_1\setminus K_1})\subset \mathbb C^2\setminus r_1\overline{\mathbb B^2}$.  Then the conditions
in Proposition \ref{firstpaper:indstep} are satisfied with $n=1$, and setting $\delta_2\le\epsilon\cdot 2^{-4}$, $r_2=2$, we let $f_2$ be the map furnished by the proposition, with 
$\epsilon_1$ explained in the induction scheme below. We then 
choose $\delta_2'<\delta_2/2$ sufficiently close to zero such that if we set 
$$
K_2:=\mathbb P^1\setminus Q_2(\delta_2'), \\
$$
we have $K_2\subset C_2^\circ$ and $f_2(\overline{C_2\setminus K_2})\subset \mathbb C^2\setminus r_2\overline{\mathbb B^2}$.  \
\medskip
\newline

Let us now state our induction hypothesis $I_n$ for some $n\geq 2$.  We assume that we have found and constructed the following. 

\begin{enumerate}[label=(\roman*)$_n$]
\item A decreasing sequence  $\delta_2>\delta_3>\cdots > \delta _n$ of numbers with $\delta_k\leq \epsilon\cdot 2^{-2k}$
such that  $\{Q_k\}_{k=1}^n$ is a nested sequence of rectangles. \label{firstpaper:ind:one}\\

\item A decreasing sequence  $\delta_1'>\delta_2'>\cdots > \delta _n'$  of numbers with $\delta'_1, \delta_2'$ as above, and 
$\delta_k'<\delta_k/2$ for $k=1,\dots,n$, and holomorphic embeddings $f_k\colon C_k\hookrightarrow\C^2$ such that, setting 
$K_k:=\mathbb P^1\setminus Q_k(\delta_k')$, we have that $f_{m}(\overline{C_{m}\setminus K_k})\subset\mathbb C^2\setminus r_k\overline{\mathbb B^2}$
for  $1\leq k\leq m\leq n$, where $r_k\geq k$. \label{firstpaper:ind:two}\\

\item A sequence of positive numbers  $\{\eta_k\}_{k=2}^n$ such that if $f\colon K_k\rightarrow\mathbb C^2$ 
is a holomorphic map with $\|f-f_k\|_{K_k}<\eta_k$, then $f\colon K_{k-1}\hookrightarrow\mathbb C^2$  is an embedding. \label{firstpaper:ind:three}\\

\item A sequence of positive numbers $\{\varepsilon_k\}_{k=1}^{n-1}$ such that $\varepsilon_{k+j}<\eta_k\cdot2^{-j-1},\; j\leq n-k-1$, with 
$\|f_{k}-f_{k-1}\|_{K_{k-1}}<\varepsilon_{k-1}$ for $k=2,\dots,n$.\label{firstpaper:ind:four}\\
\end{enumerate}

Our constructions above gives \ref{firstpaper:ind:one}, \ref{firstpaper:ind:two} and \ref{firstpaper:ind:four} in the case $n=2$ (possibly shrinking $\epsilon_1$). 
Then, choosing $\eta_2$ small enough, gives $f$ and $f'$ close to $f_2$ and $f_2'$ respectively (the latter by Cauchy estimates) on $K_2$ such that $f$ is injective and $f'$ never vanishes on $K_1$.
Being $K_1$ compact, this is enough to achieve \ref{firstpaper:ind:three} when $n=2$.

\subsection{Passing from $I_n$ to $I_{n+1}$}
Let us assume that $I_n$ is true and prove $I_{n+1}$. First of all we have that \ref{firstpaper:ind:one}$_{+1}$, 
\ref{firstpaper:ind:three}$_{+1}$ and the first part of \ref{firstpaper:ind:four}$_{+1}$ are just a matter of choosing respectively $\delta_{n+1},\;\eta_{n+1}$ and $\epsilon_n$ sufficiently small. 
By \ref{firstpaper:ind:two} with $k=n$, and with $\epsilon_n$ above fixed, we may apply Proposition \ref{firstpaper:indstep} to get a holomorphic embedding
$f_{n+1}\colon C_{n+1}\hookrightarrow\Bbb C^2$ to obtain the second part of \ref{firstpaper:ind:four}$_{+1}$ and \ref{firstpaper:ind:two}$_{+1}$ with $m=n+1$ and $k=n$.  Next, by choosing $\delta_{n+1}'$ sufficiently small we get \ref{firstpaper:ind:two}$_{+1}$ for $k=m=n+1$. It remains to explain how to achieve \ref{firstpaper:ind:two}$_{+1}$ for 
$m=n+1$ and $k=1,\dots,n-1$.  Since 
$$
\overline{C_{n+1}\setminus K_k}= \overline{C_{n+1}\setminus K_n}\cup\overline{K_n\setminus K_k}
$$
what is needed is $f_{n+1}(\overline{K_n\setminus K_k})\subset\mathbb C^2\setminus r_k\overline{\mathbb B^2}$. This follows from \ref{firstpaper:ind:two}, possibly 
after having decreased $\epsilon_n$.

\subsection{Proof of Theorem \ref{firstpaper:main}} 

Consider the objects constructed in the inductive scheme above. Then by \ref{firstpaper:ind:four} we have that 
$\lim_{j\rightarrow\infty} f_j = f$ exists on $K_k$ for any $k$. We have that 
$$
\left(\bigcup_kC_k\right)^\circ=\Bbb P^1\setminus\bigcap_k Q_k=: \mathbb P^1\setminus C
$$
and so $\lim_{j\rightarrow\infty} f_j=f$ exists on $\mathbb P^1\setminus C$. Now for any fixed $k$ we get by \ref{firstpaper:ind:two} that 
$f_n^{-1}(r_k\overline{\mathbb B^2})\subset K_k$ for all $n>k$ and therefore $f^{-1}(r_k\overline{\mathbb B^2})\subset K_k$, so $f$ is proper. 
By \ref{firstpaper:ind:four} we get that $\|f-f_k\|_{K_k}<\eta_k$, hence by \ref{firstpaper:ind:three} we have that $f\colon K_{k-1}\hookrightarrow\mathbb C^n$ is an embedding 
for all $k$, so $f$ is an embedding.  
Finally, note that when constructing a rectangle $Q_{n+1}$ from $Q_n$, a crude estimate 
gives that one obtains $Q_{n+1}$ by removing strips of total area bounded by $2^n\cdot\delta_n$.  It follows that 
$$
\mu(C) = \mu\left(\bigcap_n Q_n\right)\geq 4-\sum_{n=1}^\infty 2^n\cdot\delta_n\geq 4 - \epsilon\cdot\sum_{n=1}^\infty 2^{-n} > 4-\epsilon. 
$$

$\hfill\blacksquare$

\bibliographystyle{amsplain}

\end{document}